\documentclass[11pt]{article}
\usepackage[plainpages=false]{hyperref}
\usepackage{amsfonts,amsmath,amssymb,amsthm}
\newtheorem{theorem}{\bf Theorem}[section]
\newtheorem{Lem}[theorem]{\bf Lemma}
\newtheorem{Prop}[theorem]{\bf Proposition}
\newtheorem{Def}[theorem] {\bf Definition}

\newtheorem{Cor}[theorem]{\bf Corollary}
\theoremstyle{remark}
\newtheorem{remark} [theorem]{\bf Remark}

\theoremstyle{definition}

\makeatletter
\newcommand{\Extend}[5]{\ext@arrow0099{\arrowfill@#1#2#3}{#4}{#5}}
\makeatother

\begin{document}

 \title{ High-order  Cheeger's inequality on domain
 }
\author{{\sc Shumao Liu}\\
         { \small School of Statistics and Mathematics}\\
         { \sc\small  Central University of Finance and Economics}\\
         {\sc\small  Beijing,\ China,\ 100081}\\
         {\small (lioushumao@163.com ) }\\
         \date{}
        }

\maketitle

%\begin{abstract}
%\end{abstract}

 %\begin{center}
 %\begin{minipage}{120mm}
  % { \small {\bf Key Words:}
 %     {Hartree equation.}
  % }\\
   % { \small {\bf AMS Classification:}
    %  { 35Q40, 35Q55, 47J35.}
     % }
 %\end{minipage}
 %\end{center}

\begin{abstract}
We study the relationship of higher order variational eigenvalues of $p$-Laplacian and the higher order Cheeger constants. The asymptotic behavior of the $k$-th Cheeger constant is investigated. Using methods developed in \cite{2}, we obtain the high-order Cheeger's inequality of $p$-Laplacian on domain $h_k^p(\Omega)\leq C \lambda_{k}(p,\Omega)$.
\vspace{2.5mm}

\noindent{\textit{\textbf{Keywords}}}: {\small High order Cheeger's inequality; eigenvalue problem;
 $p$-Laplacian}
\end{abstract}

\section{\sc Introduction.}
\setcounter{section}{1}\setcounter{equation}{0}

\quad \,\,\;Let $\Omega\subset\mathbb{R}^n$ be a bounded open domain. The minimax of the so-called Rayleigh quotient
\begin{equation}\label{lambda}
\lambda_{k}(p,\Omega)=\inf_{A\in \Gamma_{k,p}}\max_{u\in A}\displaystyle \frac{\int_{\Omega}{|\nabla u|^pdx}}{\int_{\Omega}|u|^pdx},\ (1<p<\infty),
\end{equation}
leads to a nonlinear eigenvalue problem, where
$$
\Gamma_{p,k}=\{A\in  {W^{1,p}_0}(\Omega)\backslash \{0\}| A\cap\{\|u\|_p=1\}\mbox{is compact,}\ A \mbox{symmetric,}\ \gamma(A)\geq k\}.
$$
  The corresponding Euler-Lagrange equation is
\begin{equation}\label{p-laplacian equation}
 -\Delta_pu:=-\mbox{div}(|\nabla u|^{p-2}\nabla u)=\lambda|u|^{p-2}u,
\end{equation}
with Dirichlet boundary condition. This eigenvalue problem has been extensively studied in the literature.  When $p=2$, it is the familiar linear Laplacian equation$$
\Delta u+\lambda u=0.
$$ The solution of this Laplacian equation describes the shape of an eigenvibration, of frequency $\sqrt{\lambda}$, of homogeneous membrane stretched in the frame $\Omega$. It is well-known that the spectrum of Laplacian equation is discrete and all eigenfunctions form an orthonormal basis for $L^2(\Omega)$ space.
For general $1<p<\infty$, the first eigenvalue $\lambda_1(p,\Omega)$ of $p$-Laplacian $-\Delta_p$ is simple and isolated. The second eigenvalue $\lambda_2(p,\Omega)$ is well-defined and has a ``variational characterization", see \cite{20}. It has exactly 2 nodal domains, c.f.\cite{14}. However, we know little about the higher eigenvalues and eigenfunctions of the $p$-Laplacian when $p\not=2$. It is unknown  whether the variational eigenvalues (\ref{lambda}) can exhaust the spectrum  of equation (\ref{p-laplacian equation}). In this paper, we only discuss the variational eigenvalues (\ref{lambda}). For (\ref{lambda}), there are  asymptotic estimates, c.f.\cite{17} and \cite{18}. \cite{21}, \cite{22}, and \cite{23} discuss the $p$-Laplacian eigenvalue problem as $p\rightarrow\infty$ and $p\rightarrow 1$.

The Cheeger's constant which was  first studied by J.Cheeger in \cite{9}  is defined by
\begin{equation}\label{cheeger inequality}
h(\Omega):=\displaystyle\inf_{D\subseteq\Omega}\frac{|\partial D|}{|D|},\end{equation} with $D$ varying over all smooth subdomains of $\Omega$ whose boundary $\partial D$ does not touch $\partial\Omega$ and with $|\partial D|$ and $|D|$ denoting $(n-1)$ and $n$-dimensional Lebesgue measure of $\partial D$ and $D$. We call a set $C\subseteq \overline{\Omega}$ Cheeger set of $\Omega$, if $\displaystyle\frac{|\partial C|}{|C|}=h(\Omega)$. For more about the uniqueness and regularity, we refer to \cite{11}. Cheeger sets are of significant importance in the modelling of landslides, see \cite{24},\cite{25}, or in fracture mechanics, see \cite{26}.

 The classical Cheeger's inequality is about the first eigenvalue of Laplacian and the Cheeger constant(c.f.\cite{3})
$$
\lambda_{1}(2,\Omega)\geq \bigg(\frac{h(\Omega)}{2}\bigg)^2\quad\mbox{i.e.}\quad h(\Omega)\leq 2\sqrt{\lambda_{1}(2,\Omega)},
$$
which was extent to the $p$-Laplacian  in \cite{12}:
$$
\lambda_{1}(p,\Omega)\geq \bigg(\frac{h(\Omega)}{p}\bigg)^p.
$$
When $p=1$,  the first eigenvalue of $1$-Laplacian is defined by
\begin{equation}\label{1-laplace}
\lambda_{1}(1,\Omega):=\min_{0\not=u\in BV(\Omega)}\displaystyle\frac{\int_{\Omega}|Du|+\int_{\partial\Omega}|u|d\mathcal{H}^{n-1}}{\int_{\Omega}|u|dx},
\end{equation}
where $BV(\Omega)$ denotes the space of functions of bounded variation in $\Omega$.  From \cite{3}, $\lambda_{1}(1,\Omega)=h(\Omega)$. And, problem (\ref{cheeger inequality}) and problem (\ref{1-laplace}) are equivalent in the following sense: a function $u\in BV(\Omega)$ is a minimum of (\ref{1-laplace}) if and only if almost every level set is a Cheeger set. An important difference between $\lambda_1(p,\Omega)$ and $h_k(\Omega)$ is that the first eigenfunction of $p$-Laplacian is unique while the uniqueness of Cheeger set  depends  on the topology of the domain. For counterexamples, see \cite[Remark 3.13]{4}.   For more results about the eigenvalues of 1-Laplacian, we refer to \cite{6} and \cite{7}.

As to the more general Lipschitz domain, we need the following definition of perimeter:
$$
P_{\Omega}(E):=\sup\bigg{\{}\int_E \mbox{div}\phi dx\bigg{|} \phi\in C^1_c(\Omega, \mathbb{R}^n), |\phi|\leq 1, \mbox{div} \phi\in L^{\infty}(\Omega)\bigg{\}}.
$$
For convenience, we denote $|\partial E|:=P_{\Omega}(E)$. The higher order Cheeger's constant is defined by
$$
h_k(\Omega):=\inf\{\lambda\in \mathbb{R}^+|\exists \ E_1,E_2,\cdots,E_k\subseteq \Omega, E_i\cap E_j=\emptyset,\mbox{µ±}i\not=j,\max_{1,2,\cdots,k}\frac{|\partial E_i|}{|E_i|}\leq \lambda \};
$$
if$|E|=0$, we set $\displaystyle\frac{|\partial E|}{|E|}=+\infty$.
An equivalent characterization of the higher order Cheeger constant is (see\cite{4})
$$
h_k(\Omega):=\inf_{\mathfrak{D}_k}\max_{i=1,2,\cdots,k}h(E_i),
$$where $\mathfrak{D}_k$ are the set of all partitions of $\Omega$ with $k$ subsets. We set $h_1(\Omega):=h(\Omega)$. Obviously, if $R\subseteq \Omega$, then $h_k(\Omega)\leq h_k(R)$.

For the high-order Cheeger constants, there is a conjecture:
\begin{equation}\label{conjecture}
\lambda_{k}(p,\Omega)\geq \bigg(\frac{h_k(\Omega)}{p}\bigg)^p.\qquad \forall\ 1\leq k < +\infty, \ 1< p < +\infty.
\end{equation}
From \cite[Theorem 3.3]{14}, the second variational eigenfunction of $-\Delta_p$  has exactly two nodal domains, see also \cite{20}. It follows that (\ref{conjecture}) is hold for $k=1,2.$  We refer to \cite[Theorem 5.4]{4} for more details. However, by Courant's  nodal domain theorem, for other variational eigenfunctions, it is not necessary to have exactly $k$ nodal domains. Therefore, the inequality (\ref{conjecture}) on domain is still an open problem for $k>2$.

In this paper, we will get an asymptotic estimate for $h_k(\Omega)$ and establish high-order Cheeger's inequality for general $k$, and discuss the reversed inequality. To deal with the high-order Cheeger's inequality, we should give some restriction on domain.
\begin{Def} If there exists $n$-dimensional rectangle $R\subset \Omega$ and $c_1,c_2$ independent of $\Omega$, such that $c_1|R|\leq |\Omega|\leq c_2|R|$, we say   $R$  the comparable inscribed rectangle of $\Omega$.
\end{Def}

In graph theory, when $p=2$ the high-order Cheeger inequality  was  proved in \cite{1}, and was improved in \cite{2}.  In \cite{1}, using orthogonality of the eigenfunctions of Laplacian in $l_2$ and a random partitioning, they got $$
\frac{\lambda_k}{2}\leq \rho_G(k)\leq O(k^2)\sqrt{\lambda_k},
$$ where $\rho_G(k)$ is the $k$-way expansion constant, the analog of $h_k$. But, when it comes to the domain case, there is no such random partitioning. Therefore, we adapt the methods in \cite{2} to get:

\begin{theorem}\label{theorem-1}
Let $\Omega\subset\mathbb{R}^n$ be a  bounded  domain with a comparable inscribed rectangle. For $1<p<\infty$,  we have the following asymptotic estimates:
\begin{equation}\label{my Cheeger inequality}
h_k(\Omega)\leq C {k^{\frac{1}{n}}}\bigg(\frac{\lambda_1(p,\Omega)}{h_1(\Omega)}\bigg)^{\frac{q}{p}},\qquad \forall\ 1\leq k < +\infty£¬
\end{equation}where $C$  only depends on $n,p$, $\frac{1}{p}+\frac{1}{q}=1$.
\end{theorem}

There are some lower bounds about the first eigenvalue of $p$-Laplacian, see \cite{19}. There is  lower bound  by the $h_k$ when $\Omega$ be a planar domain with  finite connectivity $k$.

\begin{theorem}[\cite{8}]
Let $(S,g)$ be a Riemannian surface, and let $D\subset S$ be a domain homeomorphic to a planar domain of finite connectivity $k$. Let $F_k$ be the family of relatively compact subdomains of $D$ with smooth boundary and with connectivity at most $k$. Let $$
h_k(D):=\inf_{D'\in F_k}\frac{|\partial D'|}{|D'|},
$$where $|D'|$ is the area of $D'$ and $|\partial D'|$ is the length of its boundary. Then,
$$
\lambda_{1}(p,D)\geq \bigg(\frac{h_k(D)}{p}\bigg)^p.
$$

\end{theorem}

\begin{remark}
The results of theorem \ref{theorem-1} generalize the above theorem to more general cases.
\end{remark}

As to the reversed inequality,  if $\Omega\subset\mathbb{R}^n$ is convex,  the following lower bound (the Faber-Krahn inequality) for $h_k(\Omega)$ was proved in  \cite{4}:
\begin{equation}\label{Faber-Krahn inequality}
h_k(\Omega)\geq n\big(\frac{k\omega_n}{|\Omega|}\big)^{\frac{1}{n}},\ \forall\; k=1,2,\cdots,
\end{equation}where $\omega_n$ is the volume of the unit ball in $\mathbb{R}^n$.  Therefore
\begin{equation}\label{Faber-Krahn inequality-2}
0<h_1(\Omega)\leq h_2(\Omega)\leq\cdots\leq h_k(\Omega)\rightarrow +\infty, \mbox{as}\; k\rightarrow\infty.
\end{equation}
However, for general domain, inequalities (\ref{Faber-Krahn inequality}) and (\ref{Faber-Krahn inequality-2}) are not true at all for $p>1$. In fact, there are counter-examples in \cite{15} and \cite{16} to show that there exist domains such that $h_k(\Omega)\leq c,$ where $c$ depends only on $n$. Meanwhile, $\lambda_{k}(p,\Omega)\rightarrow +\infty$. Therefore, the reversed inequality of (\ref{my Cheeger inequality}) is not hold for general domain.

Let's  consider the convex domain. By the John ellipsoid theorem (c.f.\cite[theorem 1.8.2]{27}) and the definition of comparable inscribed rectangle, there exists  comparable inscribed rectangle $R$ for convex $\Omega$, such that $c_1|R|\leq |\Omega|\leq c_2|R|$, where $c_1,c_2$ depend only on n.

On the other hand, according to  \cite{17} and \cite{18}, for $1<p<+\infty$, there exist $C_1,C_2$ depending only on $p,n$, such that
\begin{equation}\label{weyl's asymptotic}
C_1\bigg(\frac{k}{|\Omega|}\bigg)^{\frac{1}{n}}\leq \lambda_{k}^{\frac{1}{p}}(p,\Omega)\leq C_2 \bigg(\frac{k}{|\Omega|}\bigg)^{\frac{1}{n}},\qquad \forall\ k \in\mathbb{N}.
\end{equation}

Therefore, if the domain is a bounded convex domain, combining Theorem\ref{theorem-1}, (\ref{Faber-Krahn inequality}) and (\ref{weyl's asymptotic}), the following inequality holds.

\begin{theorem}\label{theorem-2}
Let $\Omega\subset\mathbb{R}^n$ be a  bounded convex domain, then there exist $C_1,C_2$ depending only on $n$, such that
$$
C_1\bigg(\frac{k}{|\Omega|}\bigg)^{\frac{1}{n}}\leq h_{k}(\Omega)\leq C_2 \bigg(\frac{k}{|\Omega|}\bigg)^{\frac{1}{n}},\qquad \forall\ k \in\mathbb{N}.
$$
\end{theorem}

By the  two theorems above, we get bilateral estimate of $h_k(\Omega)$ with respect to $\lambda_{k}(p,\Omega)$.

\begin{Cor}\label{cor-1}
Let $\Omega\subset\mathbb{R}^n$ be a bounded convex domain. Then,  for $1<p<\infty$,  there exist $C_1,C_2$ depending only on $p,n$, such that
$$
 C_1 \lambda_{k}(p,\Omega)\leq h_k^p(\Omega)\leq C_2 \lambda_{k}(p,\Omega).
$$
\end{Cor}

\begin{remark}
From \cite{4}, when $\Omega\subset \mathbb{R}^n$ be a lipschitz domain, there is
$$
\limsup_{p\rightarrow 1}\lambda_k(p,\Omega)\leq h_k(\Omega).
$$
\end{remark}

This paper is arranged as follows: In section 2, we get some variants of Cheeger's inequalities.  Section 3 is devoted to prove Theorem \ref{theorem-1} and Theorem \ref{theorem-2}.

\vspace{5mm}

\section{\sc Some variants of Cheeger's inequalities}

\setcounter{equation}{0}

 \quad \ In this section, we will give several variants of Cheeger's inequalities. For a subset $S\subseteq \Omega$, define $\displaystyle\phi(S)=\frac{|\partial S|}{|S|}$. The  Rayleigh quotient of $\psi$ is defined by $\displaystyle\mathcal{R}(\psi):=\frac{\int_{\Omega}|\nabla\psi|^pdx}{\int_{\Omega}|\psi|^pdx}$. We define the support of $\psi$, $Supp(\psi)=\{x\in\Omega|\psi(x)\not=0\}$. If $Supp(f)\cap Supp(g)=\emptyset$, we say $f$ and $g$ are disjointly supported.
Let  $\Omega_\psi(t):=\{x\in\Omega|\psi(x)\geq t\}$ be the level set of $\psi$. For an interval $I=[t_1, t_2]\subseteqq \mathbb{R}$. $|I|=|t_2-t_1|$ denote the length of $I$. For any function $\psi$, $\Omega_\psi(I):=\{x\in \Omega|\psi(x)\in I\}$, $\displaystyle\phi(\psi):=\min_{t\in \mathbb{R}}\phi(\Omega_{\psi}(t))$. $t_{opt}:=\min\{t\in\mathbb{R}|\phi(\Omega_\psi(t))=\phi(\psi)\}$.

\begin{Lem}\label{lemma-1}
For any $\psi\in W_0^{1,p}(\Omega)$, there exist a subset  $S\subseteq Supp\psi$, such that $\phi(S)\leq p({\mathcal{R}(\psi)})^{\frac{1}{p}}$.
\end{Lem}
The proof can be found in the appendix of \cite{11}, we write it here for the reader's convenience.
\begin{proof}
Note that $|\nabla|\psi||\leq |\nabla \psi|$. We only need to show the conclusion for $\psi\geq 0$. Suppose first that $\omega\in C_0^{\infty}(\Omega)$. Then by the coarea formula and by Cavalieri's principle
$$
\int_{\Omega}|\nabla\omega|dx=\int_{-\infty}^{\infty}|\partial\Omega_\omega(t)|dt=\int_{-\infty}^{\infty}\frac{|\partial\Omega_\omega(t)|}{|\Omega_\omega(t)|}|\Omega_\omega(t)|dt
$$
\begin{equation}\label{inequality 1}\geq \inf\frac{|\partial \Omega_w(t)|}{|\Omega_w(t)|}\int_{-\infty}^{+\infty}| \Omega_w(t)|dt= \inf\frac{|\partial \Omega_w(t)|}{|\Omega_w(t)|}\int_{\Omega}|w|dx=\phi(\Omega_w{(t_{opt})})\int_{\Omega}|w|dx
\end{equation}
Since $C_0^{\infty}(\Omega)$ is dense in $W_0^{1,1}(\Omega)$, the above inequality also holds for $\omega\in W_0^{1,1}(\Omega)$ . Define $\Phi(\psi)=|\psi|^{p-1}\psi$. Then H\"{o}lder's inequality implies
$$
\int_{\Omega}|\nabla\Phi(\psi)|dx=p\int_{\Omega}|\psi|^{p-1}|\nabla\psi|dx\leq p\|\psi\|_p^{p-1}\|\nabla \psi\|_p.
$$
Meanwhile, (\ref{inequality 1}) applies and$$
\int_{\Omega}|\nabla\Phi(\psi)|\geq  \phi(\Omega_\Phi{(t_{opt})})\int_{\Omega}|\psi|^pdx.
$$Therefore, there exist a subset  $S=:\Omega_\Phi{(t_{opt})}\subseteq Supp\psi$, such that
$$
\phi(S)\leq \frac{\int_{\Omega}|\nabla\omega |dx}{\int_{\Omega}|\omega|dx}\leq \frac{p\|\psi\|_p^{p-1}\|\nabla \psi\|_p}{\int_{\Omega}|\psi|^pdx}=p\frac{\|\nabla \psi\|_p}{\|\psi\|_p}=p({\mathcal{R}(\psi)})^{\frac{1}{p}}.
$$
\end{proof}
Let $\displaystyle \mathcal{E}_f:=\int_{\Omega}|\nabla f|^pdx$. Then $\displaystyle\mathcal{R}(f)=\frac{\mathcal{E}_f}{\|f\|_p^p}$. To use the classical Cheeger's inequality for truncated functions, we introduce $\displaystyle \mathcal{E}_f(I):=\int_{\Omega_f(I)}|\nabla f|^pdx$.

\begin{Lem}\label{Lem-2}
For any function $f\in W_0^{1,p}(\Omega)$, and interval $I=[b,a]$ with $a>b\geq 0$, we have
$$
\mathcal{E}_f(I)\geq \frac{(\phi(f)|\Omega_f(a)||I|)^p}{|\Omega_f(I)|^{\frac{p}{q}}},
$$where $\frac{1}{p}+\frac{1}{q}=1$.
\end{Lem}
\begin{proof} We first prove it for $f\in C_0^\infty(\Omega)$.
By Coarea formula and Cavalieri's principle,
$$
\int_{\Omega_f(I)}|\nabla f|dx=\int_I|\partial \Omega_f(t)|dt=\int_I\frac{|\partial\Omega_f(t)|}{|\Omega_f(t)|}|\Omega_f(t)|dt
$$
$$
\qquad\qquad\qquad \geq \phi(f)\int_I|\Omega_f(t)|dt\geq \phi(f)|\Omega_f(a)||I|.
$$
The H\"{o}lder inequality gives
$$
\bigg(\int_{\Omega_f(I)}|\nabla f|dx\bigg)\leq \bigg(\int_{\Omega_f(I)}|\nabla f|^pdx\bigg)^{\frac{1}{p}}\bigg(\int_{\Omega_f(I)}dx\bigg)^{\frac{1}{q}}=\bigg(\int_{\Omega_f(I)}|\nabla f|^pdx\bigg)^{\frac{1}{p}}|\Omega_f(I)|^{\frac{1}{q}}.
$$
Combining above two inequalities, we get this Lemma for $f\in C_0^\infty(\Omega)$.
Arguments as in the proof of Lemma \ref{lemma-1} give this lemma for $f\in W_0^{1,p}(\Omega)$.
\end{proof}

\section{\sc Construction of separated functions }
\setcounter{equation}{0}

In this section, we will prove theorem \ref{theorem-1} and theorem \ref{theorem-2}. We use the method developed in \cite{2} for high-order Cheeger's inequality on graph.  Our proof consists of three steps. First, we will deal with the case of $n$-dimensional rectangle $\Omega=(a_1,b_1)\times(a_2,b_2)\times\cdots(a_n,b_n)\subset\mathbb{R}^n$ with single variable changed. Second, we extend to the case of $n$-dimensional rectangle $\Omega=(a_1,b_1)\times(a_2,b_2)\times\cdots(a_n,b_n)\subset\mathbb{R}^n$ with multi-variables changed. Finally, we discuss the general domain.

\subsection{\sc $n$-dimensional rectangle with single variable.}\label{subsection4.1}

For any non-negative function $f\in W_0^{1,p}(\Omega)$ with $\|f\|_{W_0^{1,p}}=1$.  In this subsection, we discuss $f(x_1,x_2,\cdots,x_l,\cdots,x_n)$ with $x_l$ changed and other variables unchanged. We denote $\delta:=(\frac{\phi^p(f)}{\mathcal{R}(f)})^{\frac{q}{p}}$. Given $I\subseteq \mathbb{R}^+$, let $L(I):=\displaystyle\int_{\{x\in\Omega| f(x)\in I\}}|f|^pdx.$  We say $I$ is $W$-dense if $L(I)\geq W$. For any $a\in \mathbb{R}^+$, we define
\begin{equation}\label{distance function}
dist(a,I):=\inf_{b\in I}\frac{|a-b|}{b}.
\end{equation}
The $\varepsilon$-neighborhood of a region $I$ is the set $N_\varepsilon(I):=\{a\in \mathbb{R}^+| dist(a,I)<\varepsilon\}$. If $N_\varepsilon(I_1)\cap N_\varepsilon(I_2)=\emptyset$, we say $I_1,I_2$ are $\varepsilon$-well separated.

\begin{Lem}\label{lem4.1}
Let $I_1,\cdots, I_{2k}$ be a set of $W$-dense and $\varepsilon$-well separated regions. Then, there are $k$ disjointly supported functions $f_1,\cdots,f_k$, each supported on the  $\varepsilon$-neighborhood of one of the regions such that $$
\mathcal{R}(f_i)\leq \frac{2^{p+1}\mathcal{R}(f)}{k\varepsilon^p W}, \forall 1\leq i\leq k.
$$
\end{Lem}
\begin{proof}
For any $1\leq i\leq 2k$, we define the truncated function
$$
f_i(x):=f(x)\max\{0,1-\frac{dist(f(x),I_i)}{\varepsilon}\}.
$$
Then $\|f_i\|_p^p\geq L(I_i)$.  Noting that the regions are $\varepsilon$-well separated, the functions are disjointly supported.  By an averaging argument, there exist $k$ functions $f_1,\cdots, f_k$ (after renaming) satisfy the following.
$$
 \int_{\Omega}|\nabla f_i|^pdx\leq \frac{1}{k}\sum_{j=1}^{2k}\int_{\Omega}|\nabla f_j|^pdx, \quad 1\leq i\leq k.
 $$
 By the construction of distance and $I_i\subset \mathbb{R}^1$, we know that $$
 \int_{\Omega}|\nabla \max\{0,1-\frac{dist(f(x),I_i)}{\varepsilon}\}|^p|f(x)|^pdx\leq (\frac{1+\varepsilon}{\varepsilon})^p\int_{\Omega}|\nabla f(x)|^pdx.
 $$
 Therefore $$
 \int_{\Omega}|\nabla f_i(x)|^pdx\leq (1+(\frac{1+\varepsilon}{\varepsilon})^p)\int_{\Omega}|\nabla f(x)|^pdx\leq (\frac{2}{\varepsilon})^p\int_{\Omega}|\nabla f|^pdx.
 $$
  Thus, for $1\leq i\leq k$,$$
  \mathcal{R}(f_i)=\displaystyle\frac{\|\nabla f_i\|_p^p}{\|f_i\|_p^p}\leq \frac{\displaystyle\sum_{j=1}^{2k}\int_{\Omega}|\nabla f_j|^pdx}{\displaystyle k\min_{i\in[1,2k]}\|f_i\|_p^p}\leq \frac{\displaystyle 2 (2^p\int_{\Omega}|\nabla f|^pdx)}{k\varepsilon^pW}=\frac{ 2^{p+1}\mathcal{R}(f)}{k\varepsilon^pW}.$$
\end{proof}

Let $0<\alpha<1$ be a constant that will be fixed later. For $i\in \mathbb{Z}$, we define the interval $I_i:=[\alpha^{i+1},\alpha^i]$. We let $L_i:=L(I_i)$. We partition each interval $I_i$ into $12k$ subintervals of equal length.
$$I_{i,j}=[\alpha^i(1-\frac{(j+1)(1-\alpha)}{12k}),\alpha^i(1-\frac{j(1-\alpha)}{12k})],\ \ \mbox{for}\ 0\leq j\leq 12k.$$
So that $\displaystyle |I_{i,j}|=\frac{\alpha^i(1-\alpha)}{12k}$. Set $L_{i,j}=L(I_{i,j})$. We say a subinterval $I_{i,j}$ is heavy, if $\displaystyle L_{i,j}\geq \frac{c\delta L_{i-1}}{k}$, where $c>0$ is a constant determined later. Otherwise we say it is light. We use $\mathcal{H}_i$ to denote the set of heavy subintervals of $I_i$ and $\mathcal{L}_i$ for the set of light subintervals. Let $h_i:=\sharp \mathcal{H}_i$ the number of heavy subintervals.
If $h_i\geq 6k$, we say $I_i$ is balanced, denoted by $I_i\in \mathcal{B}$.

Using Lemma \ref{lem4.1}, it is sufficient to find $2k$, $\displaystyle\frac{\delta}{k}$-dense, $\displaystyle\frac{1}{k}$ well-separated regions$R_1,R_2,\cdots R_{2k}$, such that each regions are  unions of heavy subintervals.
We will use the following strategy: from each balanced interval we choose $2k$ separated heavy subintervals and include each of them in one of the regions. In order to keep that the regions are well separated, once we include $I_{i,j}\in \mathcal{H}_i $ into a region $R$ we leave the two neighboring subintervals $I_{i,j-1}$ and $I_{i,j+1}$ unassigned, so as to separate $R$ form the rest of the regions.  In particular, for all $1\leq a\leq 2k$ and all $I_{i}\in \mathcal{B}$, we include the $(3a-1)$-th heavy subinterval of $I_i$ in $R_a$.  $R_a:=\cup_{I_i\in\mathcal{B} }I_{i,a}$.  Note that if an interval is balanced, then it has $6k$ heavy subintervals and we can include one heavy subinterval in each of the $2k$ regions. Moreover, by the construction of the distance function (\ref{distance function}),  the regions are $\frac{1-\alpha}{12k}$-well separated. It remains to prove that these $2k$ regions are dense. Let
$$
\Delta:=\sum_{I_{i}\in \mathcal{B}}L_{i-1}.
$$
Then, since each heavy subinterval $I_{i,j}$ has a mass of $\frac{c \delta L_{i-1}}{k}$, by the construction all regions are $\frac{c\delta\Delta}{k}$-dense.

Therefore, we have the following lemma.
\begin{Lem}\label{Lem3.2}
There are $k$ disjoint supported functions $f_1,\cdots, f_k$, such that for all $1\leq i\leq k$, $supp(f_i)\subseteq supp(f)$, and
$$
\mathcal{R}(f_i)\leq \bigg(\frac{24k}{(1-\alpha)}\bigg)^p\frac{2 \mathcal{R}(f)}{c\delta\Delta }, \quad \forall 1\leq i\leq k.
$$
\end{Lem}
Now we just need to lower bound $\Delta$ by an absolute constant.

\begin{Prop}
For any interval $I_i\not\in \mathcal{B}$,
$$
\mathcal{E}(I_i)\geq \frac{6(\alpha^{p(1+\frac{1}{q})}(1-\alpha))^p\phi^p(f)L_{i-1}}{(12)^p(c\delta)^{\frac{p}{q}}},
$$where $\frac{1}{q}+\frac{1}{p}=1$.
\end{Prop}
\begin{proof}
Claim: For any light interval $I_{i,j}$,$$
\mathcal{E}(I_{i,j})\geq \frac{(\alpha^{p(1+\frac{1}{q})}(1-\alpha))^p\phi^p(f)L_{i-1}}{(c\delta)^{\frac{p}{q}}(12)^pk}.
$$

Indeed, observe that $$
L_{i-1}=\int_{\Omega_f(I_{i-1})}|f(x)|^pdx\leq |\alpha^{i-1}|^p|\Omega_f(I_{i-1})|\leq  |\alpha^{i-1}|^p|\Omega_f(\alpha^i)|.
$$
Thus
$$
|\Omega_f(I_{i,j})|=\int_{\Omega_f(I_{i,j})}dx\leq \int_{\Omega_f(I_{i,j})}\frac{|f(x)|^p}{|\alpha^{i+1}|^p}dx=\frac{1}{(\alpha^{i+1})^p} \int_{\Omega_f(I_{i,j})}|f(x)|^pdx
$$
$$
\qquad =\frac{1}{(\alpha^{i+1})^p}L_{i,j}\leq \frac{c\delta L_{i-1}}{k(\alpha^{i+1})^p}\leq  \frac{c \delta |\alpha^{i-1}|^p|\Omega_f(I_{i-1})|}{k(\alpha^{i+1})^p}\leq \frac{c\delta |\Omega_f(\alpha^i)|}{k\alpha^{2p}}.
$$where we use the assumption that $I_{i,j}\in \mathcal{L}_i$. Therefore, by Lemma \ref{Lem-2},
$$
\mathcal{E}(I_{i,j})\geq \frac{(\phi(f)|\Omega_f(\alpha^i)||I_{i,j}|)^p}{|\Omega_f(I_{i,j})|^{\frac{p}{q}}}\geq \frac{(k\alpha^{2p})^{\frac{p}{q}}(\phi(f)|\Omega_f(\alpha^i)||I_{i,j}|)^p}{(c\delta |\Omega_f(\alpha^i)|)^{\frac{p}{q}}}=\frac{(k\alpha^{2p})^{\frac{p}{q}}|\Omega_f(\alpha^i)|(\phi(f)|I_{i,j}|)^p}{(c\delta)^{\frac{p}{q}}}.
$$
Note that $|I_{i,j}|=\frac{\alpha^i(1-\alpha)}{12k}$, we have
$$
\mathcal{E}(I_{i,j})\geq (\frac{k\alpha^{2p}}{c\delta})^{\frac{p}{q}}\frac{L_{i-1}}{|\alpha^{i-1}|^p}\bigg(\phi(f)\frac{\alpha^i(1-\alpha)}{12k}\bigg)^p
=\frac{L_{i-1}(\phi(f)\alpha^{p(1+\frac{1}{q})}(1-\alpha))^p}{k({c\delta})^{\frac{p}{q}}(12)^p}.
$$
Therefore, we get the Claim.

Now, since the subintervals are disjoint,
$$
\mathcal{E}(I_i)\geq \sum_{I_{i,j}\in\mathcal{L}_i}\mathcal{E}(I_{i,j})\geq (12k-h_i) \frac{L_{i-1}(\phi(f)\alpha^{p(1+\frac{1}{q})}(1-\alpha))^p}{k({c\delta})^{\frac{p}{q}}(12)^p}\geq \frac{6L_{i-1}(\phi(f)\alpha^{p(1+\frac{1}{q})}(1-\alpha))^p}{({c\delta})^{\frac{p}{q}}(12)^p},
$$where we used the assumption that $I_i$ is not balanced and thus $h_i<6k$.
\end{proof}

Now, it is time to lower-bound $\Delta$.

Note that $\|f\|_p=1$.$$
\mathcal{R}(f)=\mathcal{E}(f)\geq \sum_{I_i\not\in \mathcal{B}}\mathcal{E}(I_i)\geq \frac{6(\phi(f)\alpha^{p(1+\frac{1}{q})}(1-\alpha))^p}{({c\delta})^{\frac{p}{q}}(12)^p}\sum_{I_i\not\in \mathcal{B}}L_{i-1}.
$$
Therefore,
$$
\sum_{I_i\not\in \mathcal{B}}L_{i-1}\leq \frac{({c\delta})^{\frac{p}{q}}(12)^p\mathcal{R}(f)}{6(\phi(f)\alpha^{p(1+\frac{1}{q})}(1-\alpha))^p}.
$$
Set $\alpha=\frac{1}{2}$ and $c^{\frac{p}{q}}:=\frac{3(\alpha^{p(1+\frac{1}{q})}(1-\alpha))^p}{(12)^p}$. From the above inequality and the definition of $\delta$, we get
$$
\sum_{I_i\not\in \mathcal{B}}L_{i-1}\leq\frac{1}{2}.
$$
Note that $1=\|f\|_p^p=\sum_{I_i\in \mathcal{B}}L_{i-1}+\sum_{I_i\not\in \mathcal{B}}L_{i-1}$.  Thus, $$
\Delta=\sum_{I_i\in \mathcal{B}}L_{i-1}\geq \frac{1}{2}.
$$
 Then, by Lemma \ref{Lem3.2} and the definition of $\delta$, we get
 $$
 \mathcal{R}(f_i)\leq \frac{2\mathcal{R}(f)}{c\delta \Delta}(48k)^{p}\leq Ck^{p}\bigg(\frac{\mathcal{R}(f)}{\phi(f)}\bigg)^q.
 $$

Therefore, we have
\begin{theorem}\label{th4}
For any non-negative function $f\in W_{0}^{1,p}(\Omega)$, there are $k$ disjoint supported functions $f_1,\cdots, f_k$, such that for all $1\leq i\leq k$, $supp(f_i)\subseteq supp(f)$, and
$$
\mathcal{R}(f_i)\leq Ck^{p}\bigg(\frac{\mathcal{R}(f)}{\phi(f)}\bigg)^q, \quad \forall\ 1\leq i\leq k,
$$where $C$ depends only on $p$.
\end{theorem}
\begin{remark}
The above arguments can also be used in general dimension $n>1$ without any modification.
\end{remark}

\subsection{\sc General $n$-dimensional cases.}

 Using arguments as in above subsection, we will first discuss the case of $n$-dimensional rectangle $\Omega=(a_1,b_1)\times(a_2,b_2)\times\cdots(a_n,b_n)\subset\mathbb{R}^n$ with multi-variables changed. Then,  we deal with the general domain by comparing the volume of $\Omega$ and the inscribed rectangle.

When $\Omega$ is an $n$-dimensional rectangle $\Omega=(a_1,b_1)\times(a_2,b_2)\times\cdots(a_n,b_n)\subset\mathbb{R}^n$. we get a similar theorem as  Theorem\ref{th4}.

\begin{theorem}\label{thm4.2}
For the first eigenfunction $f\in W_{0}^{1,p}(\Omega)$, there are $k^n$ disjoint supported functions $f_{i,j}(x)$, such that for all $1\leq i\leq k, 1\leq j\leq n$, $supp(f_{i,j}(x))\subseteq supp(f(x))$, and
$$
\mathcal{R}(f_{i,j}(x))\leq Ck^{p}\bigg(\frac{\mathcal{R}(f)}{\phi(f)}\bigg)^q, \quad \forall\ 1\leq i\leq k,\ 1\leq j\leq n,
$$where $C$ depends only on $n, p$.
\end{theorem}

\begin{proof}
In the proof of Lemma \ref{lem4.1},we set $\theta_{i,j}(x)=\displaystyle\max\{0,1-\frac{dist(f(x_1,\cdots,x_j,\cdots,x_n),I_i)}{\varepsilon}\}$, where $1\leq i\leq k, 1\leq j\leq n$.
For each variable, discussing as in subsection \ref{subsection4.1}, we get $k^n$ support separated functions
$f_{i,j}(x)=f(x)\theta_{i,j}(x)$, where $1\leq i\leq k, 1\leq j\leq n$. By the construction, $supp(f_{i,j}(x))\subseteq supp(f(x))$ and $$
\mathcal{R}(f_{i,j}(x))\leq Ck^{p}\bigg(\frac{\mathcal{R}(f)}{\phi(f)}\bigg)^q, \quad \forall\ 1\leq i\leq k,\ 1\leq j\leq n,
$$where $C$ depends only on $n, p$. Therefore, we get the theorem.
\end{proof}

Finally, the case of a general bounded domain $\Omega$ with comparable inscribed n-dimensional rectangle $R\subset\Omega$, can be proved by comparison. More precisely, for the first eigenfunction $f$ of $R$, by Theorem \ref{thm4.2}, we can find $k^n$ functions $f_{i,j}(x)$. Noting that Lemma \ref{lemma-1}, we have $k^n$ subset $S_{i,j}\subset R\subset\Omega$, such that
$$
\phi(S_{i,j})\leq p(\mathcal{R}(f_{i,j}(x)))^{\frac{1}{p}}\leq C k \bigg(\frac{\mathcal{R}(f)}{\phi(f)}\bigg)^{\frac{q}{p}}.
$$
Redefining the subscript, by the definition of $h_k(\Omega)$, we have $$
h_k(\Omega)\leq h_k(R)\leq C k^{\frac{1}{n}}\bigg(\frac{\mathcal{R}(f)}{\phi(f)}\bigg)^{\frac{q}{p}}\leq C k^{\frac{1}{n}}\bigg(\frac{\lambda_{1}(p, R)}{h_1(R)}\bigg)^{\frac{q}{p}}.
$$
Therefore, we get theorem \ref{theorem-1}.

When $\Omega$ is convex,
(\ref{Faber-Krahn inequality}) and (\ref{weyl's asymptotic}) substituted into the above inequalities, we get
$$
h_k(\Omega)\leq C k^{\frac{1}{n}}\bigg(\frac{1}{|\Omega|}\bigg)^{\frac{(p-1)q}{np}}=C\bigg(\frac{k}{|\Omega|}\bigg)^{\frac{1}{n}}.
$$

Combining (\ref{Faber-Krahn inequality}) with the above inequality, we obtain Theorem \ref{theorem-2}.

Again, using (\ref{weyl's asymptotic}), there exist $C$, such that
$$
h_k^p(\Omega)\leq C \lambda_{k}(p,\Omega).
$$ Thus we prove corollary \ref{cor-1}.
\begin{center}

\end{center}

\begin{thebibliography}{99}
\addcontentsline{toc}{section}{References}

\bibitem{1} J.R. Lee, S.O. Gharan, L.Trevisan. Multi-way spectral partitioning and higher-order Cheeger inequalities. Proceedings of 44th ACM STOC. 2012, pp. 1117-1130. arXiv:1111.1055.

\bibitem{2}T.C. Kwok, L.C. Lau, Y.T. Lee, S.O. Gharan, L.Trevisan. Improved Cheeger's inequality: Analysis of spectral partitioning algorithms through
higher order spectral gap. 	  Proceedings of 45th ACM STOC. 2013, 11-20.  arXiv:1301.5584.
\bibitem{3}B.Kawohl, V. Fridman. Isoperimetric estimates for the first eigenvalue of the $p$-Laplace operator and the Cheeger constant. Comment. Math. Univ. Carolin. 44, 4(2003), 659--667.
\bibitem{4}E. Parini. The second eigenvalue of the $p$-Laplacian as $p$ goes to $1$.  Int. J. Diff. Eqns., (2010),  1--23.
\bibitem{5}E. Parini. An introduction to the Cheeger problem.  Surveys in Mathematics and its Applications, Volume 6 (2011), 9--22.
\bibitem{6}K.C. Chang. The spectrum of the $1$-Laplace operator.  Preprint.
\bibitem{7}K.C. Chang. The spectrum of the $1$-Laplace operator and Cheeger constant on graph.  Preprint.
%\bibitem{8}James R. Lee and Assaf Naor. Extending Lipschitz functions via random metric partitions.
%Invent. Math., 160(1):59¨C95, 2005. 12, 13
\bibitem{8} Osserman,R., Isoperimetric inequalities and eigenvalues of the Laplacian, Proceedings of the International Congress of Mathematics, Helsinki, (1978).
\bibitem{9}J.Cheeger. A lower bound for the smallest eigenvalue of the Laplacian, in: Problems in Analysis, A Symposium in Honor of Salomon Bochner, R.C.Gunning, Ed., Princeton Univ. Press, 1970, 195-199.
 \bibitem{10} V. Caselles; M. Novaga; A. Chambolle.  Some remarks on uniqueness and regularity of Cheeger sets,  Rendiconti del Seminario Matematico della Universit¨¤ di Padova (2010),Volume: 123, 191-202.
\bibitem{11}lefton L., Wei D., Numberical approximation of the first eigenpair of the  $p$-Laplacian using finite elements and the penalty method, Numer. Funct. Anal. Optim. 18(1997), 389--399.
\bibitem{12}Cuesta M., D.G. DE Figueiredo, J-P Gossez, A nodal domain property for the  $p$-Laplacian, C.R. Acad. Sci. Paris, t.330, S$\acute{e}$rie I,(2000), 669-673.
 \bibitem{13}  Aomar Anane, Omar Chakrone, Mohamed Filali, and Belhadj Karim, Nodal Domains for the $p$-Laplacian, Advances in Dynamical Systems and Applications,Volume 2 Number 2 (2007), 135-140.

  \bibitem{14} P. Dr$\acute{a}$bek and S. B. Robinson, On the Generalization of the Courant Nodal Domain Theorem, Journal of Differential Equations 181,(2002) 58-71.
 \bibitem{15} Jie Xiao, The p-Faber-Krahn inequality noted, Around the Research of Vladimir Maz'ya I  International Mathematical Series Volume 11, 2010,  373-390
\bibitem{16} Vladimir Maz'ya, Integral and isocapacitary inequalities,  Linear and complex analysis, Amer. Math. Soc. Transl. Ser. Amer. Math. Soc., Providence, RI vol 226(2009) issue 2, 85-107
\bibitem{17} L. Friedlander, Asymptotic  behaviour of the eigenvalues of the $p$-laplacian, Communications in Partial Differential Equations, Volume 14, Issue 8-9, 1989, 1059-1069.
   \bibitem{18} J.P.Garc\'{i}a£¬ I.Peral, Comportement asymptotique des valeurs propres du $p$-laplacien  [Asymptotic behaviour of the eigenvalues of the p-Laplacian], C.R.Acad. Sci. Paris Ser.I Math. 307 (1988),75-78.

  \bibitem{19}  G. Poliquin, Bounds on the principal frequency of the $p$-Laplacian, Arxiv:1304.5131v2 [math.SP] 31 Jan 2014.
 \bibitem{20} Anane J A, Tsouli N. On the second eigenvalue of the p-Laplacian, Nonlinear partial differential equaitons (F\'{e}s,1994). Pitman Research Notes in Mathematics Series,343,Longman,Harlow, (1996) 1-9.
  \bibitem{21}  Lindqvist, Peter; Juutinen, Petri. On the higher eigenvalues for the $\infty$-eigenvalue problem, Calculus of Variations and Partial Differential Equations, June 2005, Volume 23, Issue 2, pp 169-192.
  \bibitem{22}Belloni M, Juutinen P, Kawohl B. The $p$-Laplace eigenvalue problem as $p\rightarrow\infty$ in a Finsler metric. J. Eur. Math. Soc. 8 (2006), 123-138.
   \bibitem{23}    B. Kawohl, and M. Novaga: The $p$-Laplace eigenvalue problem as $p \rightarrow 1$ and Cheeger sets in a Finsler metric, J. Convex Anal. 15 (2008),  623-634.

 \bibitem{24} P. Hild; I. R. Ionescu; Th. Lachand-Robert; I. Rosca, The blocking of an inhomogeneous Bingham fluid. Applications to landslides,ESAIM: Mathematical Modelling and Numerical Analysis (2010) Volume 36, Issue 6, page 1013-1026.
 \bibitem{25}  I. R. Ionescu; Th. Lachand-Robert, Generalized Cheeger sets related to landslides, Calculus of Variations and Partial Differential Equations, (2005), Volume 23, Issue 2, pp 227-249.
     \bibitem{26} J.B. Keller, Plate failure under pressure, SIAM Review, 22(1980), pp. 227-228.
 \bibitem{27} C.E. Guti\'{e}rrez, The  Monge-Amp\`{e}re equation, Birkh\"{a}user, Boston. Basel. Berlin, (2001)

\end{thebibliography}
\end{document}